\documentclass[11pt]{article}

\usepackage[T1]{fontenc}
\usepackage[utf8]{inputenc}
\usepackage{lmodern}
\IfFileExists{microtype.sty}{\usepackage{microtype}}{}
\usepackage[a4paper,margin=1.05in]{geometry}
\usepackage{amsmath,amssymb,amsthm}
\usepackage{booktabs}
\usepackage{xcolor}
\usepackage{url}
\usepackage[numbers,sort&compress]{natbib}
\usepackage[colorlinks=true,linkcolor=blue!55!black,citecolor=blue!55!black,
            urlcolor=blue!55!black]{hyperref}
\hypersetup{
  pdftitle={Exact Periodicity, Surjectivity, and a Haar Limit Law
            for a Restarting Josephus Process},
  pdfauthor={Lizhong Chen},
  pdfsubject={Arithmetic and probabilistic structure of a restarting
              Josephus process}
}
\newif\ifpaperfallbackcref
\IfFileExists{cleveref.sty}{%
  \usepackage[nameinlink,noabbrev]{cleveref}%
}{%
  \paperfallbackcreftrue
}
\ifpaperfallbackcref
\makeatletter
\newcommand{\cref}[1]{%
  \expandafter\paper@cref@split#1:\@nil{#1}}
\def\paper@cref@split#1:#2\@nil#3{%
  \@ifundefined{paper@cref@#1}
    {\ref{#3}}
    {\csname paper@cref@#1\endcsname{#3}}}
\newcommand{\Cref}[1]{%
  \expandafter\paper@Cref@split#1:\@nil{#1}}
\def\paper@Cref@split#1:#2\@nil#3{%
  \@ifundefined{paper@Cref@#1}
    {\ref{#3}}
    {\csname paper@Cref@#1\endcsname{#3}}}
\def\paper@cref@eq#1{equation~\eqref{#1}}
\def\paper@cref@lem#1{lemma~\ref{#1}}
\def\paper@cref@prop#1{proposition~\ref{#1}}
\def\paper@cref@cor#1{corollary~\ref{#1}}
\def\paper@cref@thm#1{theorem~\ref{#1}}
\def\paper@cref@prob#1{open problem~\ref{#1}}
\def\paper@cref@sec#1{section~\ref{#1}}
\def\paper@cref@app#1{appendix~\ref{#1}}
\def\paper@Cref@eq#1{Equation~\eqref{#1}}
\def\paper@Cref@lem#1{Lemma~\ref{#1}}
\def\paper@Cref@prop#1{Proposition~\ref{#1}}
\def\paper@Cref@cor#1{Corollary~\ref{#1}}
\def\paper@Cref@thm#1{Theorem~\ref{#1}}
\def\paper@Cref@prob#1{Open Problem~\ref{#1}}
\def\paper@Cref@sec#1{Section~\ref{#1}}
\def\paper@Cref@app#1{Appendix~\ref{#1}}
\makeatother
\fi

\newtheorem{theorem}{Theorem}[section]
\newtheorem{maintheorem}{Theorem}

\newtheorem{proposition}[theorem]{Proposition}
\newtheorem{lemma}[theorem]{Lemma}
\newtheorem{corollary}[theorem]{Corollary}
\newtheorem{problem}[theorem]{Open Problem}
\theoremstyle{definition}
\newtheorem{definition}[theorem]{Definition}
\newtheorem{remark}[theorem]{Remark}

\ifpaperfallbackcref\else
  \crefname{maintheorem}{theorem}{theorems}
  \Crefname{maintheorem}{Theorem}{Theorems}
\fi

\DeclareMathOperator{\lcm}{lcm}
\DeclareMathOperator{\Per}{Per}
\DeclareMathOperator{\Law}{Law}
\DeclareMathOperator{\Cov}{Cov}

\DeclareMathOperator{\nextprime}{nextprime}
\newcommand{\ind}{\mathbf 1}

\title{\Large\bfseries
Exact Periodicity, Surjectivity, and a Haar Limit Law\\
for a Restarting Josephus Process}
\author{%
Lizhong Chen\thanks{Corresponding author. Email:
\texttt{lchendh@connect.ust.hk}}\\
\small Department of Mathematics\\[-1mm]
\small The Hong Kong University of Science and Technology\\[-1mm]
\small Clear Water Bay, Kowloon, Hong Kong}
\date{}

\begin{document}

\maketitle

\begin{abstract}
We study a \emph{restarting Josephus process} in which the participants retain
their linear order and counting restarts at the current leftmost survivor
after every deletion.  For step size \(m\), put \(q=m-1\), and let \(F_n(q)\)
denote the initial position of the survivor.  Reverse insertion gives
\[
 F_1(q)=1,\qquad
 F_k(q)=F_{k-1}(q)+
 \ind_{\{q\bmod k<F_{k-1}(q)\}}.
\]
Writing \(L_n=\lcm(1,\ldots,n)\), we establish three results for the compatible
residue system in this recurrence.  First, the full period group of \(F_n\) is
exactly \(L_n\mathbb Z\).  Second, \(F_n\) is surjective onto
\(\{1,\ldots,n\}\).  The proof is constructive and unconditional but
computer-assisted: a Chinese-remainder construction and explicit prime
estimates reduce it to a finite exact certificate.  Third, if
\(\widetilde Q_n\) is uniform modulo \(L_n\), then
\((F_n(\widetilde Q_n)-1)/(n-1)\) converges to a symmetric, nondegenerate law
on \([0,1]\).  A common Haar coupling yields almost-sure and \(L^r\)
convergence for every \(1\le r<\infty\), together with an \(O(n^{-1/4})\)
bound in \(W_1\).  Logarithmic boundary-mass estimates rule out every
symmetric beta law.  We also formulate endpoint dominance as an open problem,
prove strict dominance over the two nearest internal positions for every
\(n\ge4\), exclude prime levels as minimal counterexamples, and verify the
claim exactly through \(n=49\).
\end{abstract}

\medskip
\noindent\textbf{Keywords.}
Josephus problem; least common multiple; Chinese remainder theorem;
surjectivity; survivor distribution; profinite integers; Haar measure.

\medskip
\noindent\textbf{2020 Mathematics Subject Classification.}
Primary 05A05, 60F15; Secondary 11A07, 11N05, 60B10.

\section{Introduction}
\label{sec:introduction}

The Josephus problem is a classical counting-out problem.  It has led to
many questions about elimination orders, permutations, recurrences, and
algorithms.  Standard references include
\citet[Section~1.3]{GrahamKnuthPatashnik1994}, the generalized formulation of
\citet{Jakobczyk1973}, the algorithmic work of \citet{Lloyd1983}, the
functional-iteration viewpoint of \citet{OdlyzkoWilf1991}, and the explicit
formulas and bounds of \citet{HalbeisenHungerbuhler1997}.  Historical and
expository accounts may be found in \citet{Schumer2002,Groer2003}.
\citet{Theriault2000} developed algorithms for later elimination times in
the continuing-count process and for a different repeated-table variation.

In the classical circular convention, counting continues from the next
surviving participant after each deletion.  In this paper, we study a
different rule.  The participants retain their left-to-right order, and
counting restarts at the current leftmost survivor after \emph{every}
deletion.  Thus, if the step size is \(m\) and
\[
 q=m-1,
\]
then at population size \(k\ge2\) the deleted current rank is
\[
 1+(q\bmod k).
\]
We call this rule the restarting Josephus process.  It replaces the standard
modular-rotation recurrence by the threshold recurrence
\begin{equation}
 F_1(q)=1,\qquad
 F_k(q)=F_{k-1}(q)+
 \ind_{\{q\bmod k<F_{k-1}(q)\}}.
 \label{eq:intro-recurrence}
\end{equation}
Here \(F_n(q)\) is the initial position of the survivor.  The diagonal
case \(m=n\) of the same rule is recorded as OEIS A128982
\citep{OEISA128982}; thus the restart convention itself is not new.  We
study the arithmetic structure of the full two-parameter family \((n,m)\),
including its exact periods, fibers, and limiting behaviour.  To the best of
our knowledge, the exact period group,
surjectivity theorem, and limit law below have not previously been established
for this restarting model.  This process is different from the usual
continuing-count process.

For \(n\ge2\), we call
\[
 \mathbf c_n(q)
 =(q\bmod 2,q\bmod 3,\ldots,q\bmod n)
 \in\prod_{k=2}^n\mathbb Z/k\mathbb Z
\]
the \emph{residue vector of \(q\) through level \(n\)}.  Its components are
not freely prescribed.  We call a tuple \((c_2,\ldots,c_n)\) a
\emph{compatible residue vector} if
\[
 c_i\equiv c_j\pmod{\gcd(i,j)}
 \qquad(2\le i,j\le n).
\]
Every \(\mathbf c_n(q)\) is compatible.  A single class \(q\bmod L_n\),
where
\[
 L_n=\lcm(1,2,\ldots,n)
\]
determines the entire vector.  We use this compatibility condition
throughout the paper.  Translational symmetries determine the period group of
\(F_n\).  Suitable compatible residue vectors determine which survivor
labels can be reached, while a Haar-distributed point of
\(\widehat{\mathbb Z}\) gives a common coupling for the finite
distributions.  Thus the same arithmetic description is used in the proofs
of periodicity, surjectivity, and the limit law.

Chinese-remainder methods have also played an important role for classical
Josephus permutations.  In particular,
\citet{DowdyMays1989} gave a Chinese-remainder criterion for deciding whether
a prescribed elimination order arises from some step size, and
\citet{WilsonMorgan2010} used Fourier analysis on finite abelian groups for a
related enumeration.  Those works provide related arithmetic precedents, but
they concern the continuing-count process.

\subsection*{Main results}

Our three main results are as follows.

\paragraph{\Cref{thm:exact-period}: exact periodicity.}
We first determine the exact period group.  It is immediate that \(L_n\) is
a period, but this does not exclude a smaller one.  We prove that
\[
 \Per(F_n)=L_n\mathbb Z.
\]
The proof uses a trajectory that becomes permanently constant.  Together
with the least recurrence level not dividing a proposed shift, this gives an
explicit witness against every nonmultiple of \(L_n\).

\paragraph{\Cref{thm:surjectivity}: surjectivity.}
We next prove that the survivor map has image
\[
 \{F_n(q):q\in\mathbb Z\}=\{1,\ldots,n\}
\]
for every \(n\).  We use a Chinese-remainder construction with selected
residues at primes greater than \(n/2\) to produce intervals of survivor
labels.  A short-prime-interval condition joins these intervals.  The
explicit estimates of \citet{Dusart2010} handle the infinite tail, and an
exact-arithmetic program certifies the remaining finite range.  Hence the
theorem is unconditional but computer-assisted, and the finite certificate
is part of the proof.

\paragraph{\Cref{thm:limit}: a Haar limit law.}
Finally, we prove a limit theorem.
For \(n\ge2\), let \(\widetilde Q_n\) be uniform modulo \(L_n\), and define
\[
 \mu_n=\Law\!\left(\frac{F_n(\widetilde Q_n)-1}{n-1}\right).
\]
We place all finite models on a single probability space by taking a random
profinite integer \(Q\in\widehat{\mathbb Z}\) distributed according to
normalised Haar measure.  A generalized-CRT covariance identity, a
short-block uniformity estimate, and pathwise control of the random
threshold imply that
\[
 \frac{S_n-1}{n-1}\longrightarrow Y,
\]
where \(S_n=F_n(Q\bmod L_n)\),
almost surely and in every finite \(L^r\), with
\[
 \mathbb E\left|
 \frac{S_n-1}{n-1}-Y
 \right|=O(n^{-1/4}).
\]
The law \(\mu=\Law(Y)\) is symmetric and nondegenerate, and
\(W_1(\mu_n,\mu)=O(n^{-1/4})\).  Probabilistic Josephus variants with a
different random deletion mechanism were recently studied by
\citet{AdiceamEtAl2024}; the randomness here instead comes from the finite
input \(\widetilde Q_n\), chosen uniformly over a complete period.

The exact endpoint events have mass
\[
 \frac{\varphi(L_n)}{L_n}
 =\prod_{p\le n}\left(1-\frac1p\right)
 \sim\frac{e^{-\gamma}}{\log n}.
\]
Transferring this mass through the quantitative coupling yields logarithmic
lower bounds for the boundary mass of \(\mu\).  In particular, the limit is
not a symmetric beta distribution: the proof excludes every
\(\operatorname{Beta}(\alpha,\alpha)\), \(\alpha>0\).
Explicit identification of \(\mu\), non-atomicity, and the existence of a
density remain open.

\subsection*{An open structural problem}

We also consider a finer question about the fibers of the survivor map.  For
\[
 N_{n,j}=\#\{0\le q<L_n:F_n(q)=j\},
\]
reflection and the endpoint criteria give
\[
 N_{n,j}=N_{n,n+1-j},\qquad
 N_{n,1}=N_{n,n}=\varphi(L_n).
\]
We say that \emph{endpoint dominance} holds at level \(n\) if
\[
 N_{n,j}<\varphi(L_n)\qquad(1<j<n),
\]
so that the endpoint fibers are the unique global maxima.  We ask whether
this holds for every \(n\ge4\).  We prove the inequality for the
nearest internal positions \(j=2,n-1\), show that a prime level cannot be a
minimal counterexample, and derive exact identities for transitions between
adjacent survivor labels.
Exact computation verifies endpoint dominance, as well as a stronger
composite-level majorization, through \(n=49\).  We state these finite results
as evidence rather than as a proof for all \(n\).

\subsection*{Organization}

\Cref{sec:setup} develops the recurrence, reflection, endpoint criteria, and
residue-vector formulation.  \Cref{sec:period} proves
\Cref{thm:exact-period}.  \Cref{sec:range} proves
\Cref{thm:surjectivity} and records the precise role of the finite
certificate.  \Cref{sec:limit} proves \Cref{thm:limit} and its boundary
properties.  \Cref{sec:endpoint} formulates endpoint dominance and presents
the rigorous and computational partial results.  Reproduction details for
the computer-assisted components are collected in
\Cref{app:certificates}.

\section{The process and its arithmetic encoding}
\label{sec:setup}

Throughout, \(q\bmod k\) denotes the least nonnegative residue in
\(\{0,\ldots,k-1\}\), even when \(q<0\).  This convention extends the
survivor function from the original parameters \(q=m-1\ge0\) to every
integer \(q\).

\begin{definition}
The \emph{restarting Josephus process with parameter \(q\)} starts with the
ordered list \(1,\ldots,n\).  While the current list has
length \(k\ge2\), delete its element of current rank
\[
 1+(q\bmod k)
\]
from the left.  Preserve the relative order of the remaining entries and
restart counting at their leftmost member.  The initial position of the last
remaining participant is denoted by \(F_n(q)\).  We also set
\[
 L_n=\lcm(1,\ldots,n),\qquad L_1=1.
\]
\end{definition}

We use Euler's totient function with the convention \(\varphi(1)=1\).

\subsection{Reverse-insertion recurrence}

\begin{lemma}[Fundamental recurrence]
\label{lem:recurrence}
For every \(q\in\mathbb Z\),
\begin{equation}
 F_1(q)=1,\qquad
 F_k(q)=F_{k-1}(q)+
 \ind_{\{q\bmod k<F_{k-1}(q)\}}
 \quad(k\ge2).
 \label{eq:recurrence}
\end{equation}
In particular, \(1\le F_k(q)\le k\), and \(k\mapsto F_k(q)\) is
nondecreasing.
\end{lemma}

\begin{proof}
Fix a population size \(k\ge2\), and write \(r=q\bmod k\).  The current rank
deleted at this stage is \(e=r+1\).  Relabel the shortened list by the ranks
\(1,\ldots,k-1\).  Since counting restarts at the leftmost survivor, the
remaining process is a restarting process on \(k-1\) entries.

Let \(j=F_{k-1}(q)\) be the rank of the eventual survivor in the shortened
list.  When the deleted position is restored, this rank changes from \(j\)
to \(j+1\) exactly when \(e\le j\), or equivalently when \(r<j\).  This
proves \eqref{eq:recurrence}.
\end{proof}

\begin{remark}
The threshold in \eqref{eq:recurrence} is strict because the deleted rank is
\(1+(q\bmod k)\).  This is also where the restarting process differs from
ordinary Josephus continuation.
\end{remark}

The recurrence also has a permutation interpretation.  Start with the
one-letter word \([1]\).  At stage \(k\), insert the new maximum letter \(k\)
in the zero-based position \(q\bmod k\).  The position of the letter \(1\),
counted from \(1\), changes by the indicator in \eqref{eq:recurrence}.
Thus \(F_n(q)\) is the position of \(1\) in the resulting permutation.  We
shall use this reverse-insertion model again when interpreting the limiting
random variable.

\subsection{Periodicity, reflection, and endpoints}

\begin{proposition}
\label{prop:elementary}
For every \(n\ge1\) and \(q\in\mathbb Z\), the following hold.
\begin{align}
 F_n(q+L_n)&=F_n(q),                                      \label{eq:period-easy}\\
 F_n(q)+F_n(-1-q)&=n+1.                                  \label{eq:reflection-int}
\end{align}
Equivalently, for \(0\le q<L_n\),
\begin{equation}
 F_n(q)+F_n(L_n-1-q)=n+1.                                \label{eq:reflection}
\end{equation}
Moreover,
\begin{align}
 F_n(q)=1
 &\quad\Longleftrightarrow\quad \gcd(q,L_n)=1,             \label{eq:left-endpoint}\\
 F_n(q)=n
 &\quad\Longleftrightarrow\quad \gcd(q+1,L_n)=1.           \label{eq:right-endpoint}
\end{align}
\end{proposition}

\begin{proof}
Every modulus \(k\le n\) divides \(L_n\).  Hence replacing \(q\) by \(q+L_n\)
preserves every residue in \eqref{eq:recurrence}, which proves
\eqref{eq:period-easy}.

We next prove the reflection identity.  Put \(q^\ast=-1-q\).  If
\(r=q\bmod k\), then
\[
 q^\ast\bmod k=k-1-r.
\]
Suppose inductively that the two states at level \(k-1\) are \(s\) and
\(k-s\).  Their increment tests at level \(k\) are
\[
 r<s
 \quad\text{and}\quad
 k-1-r<k-s.
\]
The second condition is equivalent to \(r\ge s\).  Thus exactly one
trajectory increments.  Since \(F_1(q)=F_1(q^\ast)=1\), induction gives
\eqref{eq:reflection-int}.  Reducing modulo \(L_n\) gives
\eqref{eq:reflection}.

It remains to prove the endpoint identities.  The state starts at \(1\) and
only increases.  Thus \(F_n(q)=1\) exactly when no increment occurs.  In
this case, the state before every level is \(1\), and the condition is
\[
 q\bmod k\ne0\qquad(2\le k\le n).
\]
This is equivalent to \(\gcd(q,L_n)=1\).

Similarly, \(F_n(q)=n\) exactly when every level increments.  The state
before level \(k\) is then \(k-1\), and the increment condition becomes
\[
 q\bmod k<k-1
 \quad\Longleftrightarrow\quad
 q\bmod k\ne k-1
 \quad\Longleftrightarrow\quad
 k\nmid q+1.
\]
This condition holds for every \(2\le k\le n\) if and only if
\(\gcd(q+1,L_n)=1\).  This completes the proof.
\end{proof}

\begin{definition}
For \(1\le j\le n\), define the \emph{fiber size} of \(j\) over one complete
\(L_n\)-block by
\[
 N_{n,j}=\#\{0\le q<L_n:F_n(q)=j\}.
\]
\end{definition}

\begin{corollary}
\label{cor:counts}
For every \(n\ge1\),
\begin{align}
 N_{n,j}&=N_{n,n+1-j},                                  \label{eq:count-symmetry}\\
 N_{n,1}=N_{n,n}&=\varphi(L_n),                          \label{eq:endpoint-counts}\\
 \frac1{L_n}\sum_{q=0}^{L_n-1}F_n(q)&=\frac{n+1}{2}.      \label{eq:mean}
\end{align}
\end{corollary}

\begin{proof}
By \eqref{eq:reflection}, the involution \(q\mapsto L_n-1-q\) gives
\eqref{eq:count-symmetry} and \eqref{eq:mean}.  The endpoint criteria count
the units modulo \(L_n\).  Since translation by \(1\) is a bijection modulo
\(L_n\), they also give \eqref{eq:endpoint-counts}.
\end{proof}

\subsection{Compatible residue vectors}

For \(2\le k\le n\), write \(c_k(q)=q\bmod k\), so that
\[
 \mathbf c_n(q)=(c_2(q),\ldots,c_n(q)).
\]
This residue vector determines the deletion sequence and hence \(F_n(q)\).
Its components are not independent.

By the general Chinese remainder theorem \citep{Ore1952}, a tuple
\[
 (c_2,\ldots,c_n)\in\prod_{k=2}^n\mathbb Z/k\mathbb Z
\]
is induced by an integer \(q\) if and only if
\begin{equation}
 c_i\equiv c_j\pmod{\gcd(i,j)}
 \qquad(2\le i,j\le n).
 \label{eq:crt-compat}
\end{equation}
In that case \(q\) is unique modulo \(L_n\).  Hence the induced tuples are
exactly the compatible residue vectors defined in \Cref{sec:introduction},
and there are \(L_n\) such vectors.

We shall use the compatibility condition \eqref{eq:crt-compat} repeatedly.
In particular, the thresholds \(q\bmod k\) cannot be treated as independent
random variables when the moduli have common factors.

\section{Exact periodicity}
\label{sec:period}

By \eqref{eq:period-easy}, \(L_n\) is a period.  We now show that no smaller
positive shift is a period.  The main idea is to construct a trajectory that
becomes constant.

For fixed \(q\), we say that the trajectory \(k\mapsto F_k(q)\)
\emph{freezes at level \(u\)} if
\[
 F_k(q)=F_u(q)\qquad(k\ge u).
\]

\begin{lemma}[Freezing]
\label{lem:freezing}
If \(u\) is a positive integer, then
\[
 F_k(u)=F_u(u)\qquad(k\ge u).
\]
\end{lemma}

\begin{proof}
At level \(u\), the state \(F_u(u)\) is at most \(u\).  For every \(k>u\),
we have \(u\bmod k=u\).  Hence the strict threshold in
\eqref{eq:recurrence} is false as long as the state is at most \(u\).
Therefore, the state does not change after level \(u\).
\end{proof}

\begin{remark}
We state the lemma only for \(u\ge1\).  It cannot be formulated at \(u=0\)
because \(F_0\) is undefined.  On the other hand, \eqref{eq:recurrence}
gives \(F_k(0)=k\) for \(k\ge1\), so no analogous freezing occurs.
\end{remark}

\begin{lemma}
\label{lem:shift-detector}
Fix \(n\ge2\).  Let \(D>0\) satisfy \(L_n\nmid D\), and define the least
nondividing level
\[
 H=\min\{k\in\{2,\ldots,n\}:k\nmid D\}.
\]
Let
\[
 q=(-D)\bmod H,
\]
represented in \(\{0,\ldots,H-1\}\).  Then \(1\le q\le H-1\) and
\[
 F_n(q+D)\ne F_n(q).
\]
\end{lemma}

\begin{proof}
Since \(H\nmid D\), the residue \(q\) is nonzero.  By the choice of \(H\),
\[
 k\mid D\qquad(1\le k<H).
\]
Thus \(q\) and \(q+D\) have the same residue at every recurrence level below
\(H\).  By \eqref{eq:recurrence},
\[
 F_{H-1}(q)=F_{H-1}(q+D)=:x.
\]
Because \(1\le q\le H-1\), we may apply \cref{lem:freezing} with
\(u=q\) and \(k=H-1\).  Hence
\[
 x=F_q(q),\qquad 1\le x\le q.
\]

At level \(H\), however, the two residues are
\[
 q\bmod H=q,\qquad (q+D)\bmod H=0.
\]
The unshifted threshold \(q<x\) is false, including the boundary case
\(x=q\), whereas the shifted threshold \(0<x\) is true.  Hence
\[
 F_H(q)=x,\qquad F_H(q+D)=x+1.
\]
By \cref{lem:freezing}, the first trajectory remains equal to \(x\) at every
later level.  The second trajectory can only stay fixed or increase, so it
remains at least \(x+1\).  Hence the two trajectories cannot meet again.
\end{proof}

\begin{maintheorem}[Exact period group]
\label{thm:exact-period}
For every \(n\ge1\),
\[
 \Per(F_n)
 :=\{D\in\mathbb Z:F_n(q+D)=F_n(q)\text{ for every }q\in\mathbb Z\}
 =L_n\mathbb Z.
\]
In particular, the least positive period of \(F_n\) is \(L_n\).
\end{maintheorem}

\begin{proof}
For \(n=1\), the result follows from \(F_1\equiv1\) and \(L_1=1\).  Let
\(n\ge2\).  Every multiple of \(L_n\) preserves all residues in
\eqref{eq:recurrence}.  Hence \(L_n\mathbb Z\subseteq\Per(F_n)\).

Conversely, suppose that \(D>0\) is not divisible by \(L_n\).  By
\cref{lem:shift-detector}, there is a \(D\)-dependent witness \(q\) such that
\(F_n(q+D)\ne F_n(q)\).  Thus \(D\) is not a period.

It remains to consider negative shifts.  We first observe that the periods
form an additive subgroup of \(\mathbb Z\).  The shift \(0\) is a period.  If
\(s\) and \(t\) are periods, then
\[
 F_n(q+s+t)=F_n(q+s)=F_n(q),
\]
so \(s+t\) is a period.  If \(t\) is a period, substituting \(q-t\) for
\(q\) in its period identity gives \(F_n(q)=F_n(q-t)\), so \(-t\) is also
a period.  Thus a negative integer \(D\) is a period if and only if
\(-D\) is a period.  Applying the positive-shift result to \(-D\), we
conclude that no integer outside \(L_n\mathbb Z\) is a period.
\end{proof}

\begin{remark}
The witness \(q\) may depend on the proposed shift \(D\).  This is sufficient
because a period must preserve \(F_n(q)\) for every \(q\), whereas one
witness is enough to disprove it.  The minimality of \(H\), the sign
\(q=(-D)\bmod H\), and the strict threshold are all essential.  The proof
uses neither independence nor a Chinese-remainder assumption.
\end{remark}

\section{Surjectivity of the survivor map}
\label{sec:range}

\begin{maintheorem}[Surjectivity]
\label{thm:surjectivity}
For every integer \(n\ge1\) and every \(1\le j\le n\), there exists
\(q\in\{0,\ldots,L_n-1\}\) such that
\[
 F_n(q)=j.
\]
Equivalently,
\[
 \{F_n(q):q\in\mathbb Z\}
 =\{F_n(q):0\le q<L_n\}
 =\{1,\ldots,n\}.
\]
\end{maintheorem}

We prove the theorem in four steps.  We first give a prime-power CRT
construction and derive a prime-interval covering criterion.  We then verify
the criterion analytically for large \(n\) and use a finite exact certificate
for the remaining cases.

\subsection{A prime-power CRT interval construction}

For each prime \(p\le n\), let
\[
 M_p=p^{e_p}\le n<pM_p
\]
be the largest power of \(p\) not exceeding \(n\).  Then
\[
 L_n=\prod_{p\le n}M_p,
\]
and the moduli \(M_p\) are pairwise coprime.  By the ordinary Chinese
remainder theorem, a class modulo \(L_n\) is uniquely specified by the
residues \(q\bmod M_p\).  We call these residues the \emph{prime-power CRT
components} of \(q\).

For a fixed \(q\), call a recurrence stage \(k\ge2\) a
\emph{blocking stage} if
\[
 q\bmod k\ge F_{k-1}(q),
\]
so that the state does not increment at stage \(k\).  A prime \(\ell\) is
called a \emph{blocking prime} if stage \(\ell\) is blocking.

\begin{lemma}[CRT interval construction with blocking primes]
\label{lem:blocking-prime-interval}
Let \(R,P\) be primes satisfying
\[
 \frac n2<R\le n,\qquad P\le R,
\]
and put \(a=P-1\).  Let \(\mathcal B\) be any set of primes in
\((R,n]\), and write \(t=|\mathcal B|\).  Prescribe the prime-power CRT
components
\begin{align}
 q&\equiv0\pmod R,                                      \label{eq:blocking-R}\\
 q&\equiv-1\pmod\ell\qquad(\ell\in\mathcal B),           \label{eq:blocking-B}\\
 q&\equiv a\pmod{M_p}
 \quad\text{for every other prime }p\le n.              \label{eq:blocking-default}
\end{align}
There is a unique solution \(q\bmod L_n\), and it satisfies
\begin{align}
 F_n(q)&=n-R+P-t,                                       \label{eq:blocking-output}\\
 F_n(L_n-1-q)&=R-P+1+t.                                 \label{eq:blocking-reflected}
\end{align}
\end{lemma}

\begin{proof}
Since \(R>n/2\), the largest \(R\)-power at most \(n\) is \(R\) itself.  The
same holds for every \(\ell\in\mathcal B\).  Hence
\eqref{eq:blocking-R}--\eqref{eq:blocking-default} prescribe one residue for
each pairwise coprime prime-power modulus, and the ordinary Chinese remainder
theorem gives a unique class modulo \(L_n\).

We first establish a projection property.  Let \(k\le n\) be neither \(R\)
nor an element of \(\mathcal B\).  Then neither \(R\) nor any
\(\ell\in\mathcal B\) divides \(k\), since otherwise \(k\) would be at least
twice a prime exceeding \(n/2\).  Therefore, every prime-power modulus in the
factorisation of \(k\) divides an unmodified \(M_p\).  By
\eqref{eq:blocking-default},
\begin{equation}
 q\equiv a\pmod k.
 \label{eq:projection-a}
\end{equation}

For \(2\le k\le a=P-1\), the equality \(a\bmod k=k-1\) would imply
\(k\mid a+1=P\), impossible because \(P\) is prime and \(k<P\).  Hence
\[
 a\bmod k\le k-2.
\]
By induction in \eqref{eq:recurrence}, every level \(2,\ldots,a\) increments,
so \(F_a(q)=a\).  Next, let \(a<k<R\).  By
\eqref{eq:projection-a}, \(q\bmod k=a\), which is equal to the current state.
Thus the strict threshold is false, and
\[
 F_{R-1}(q)=a.
\]
If \(P=2\), then \(a=1\) and the first interval is empty; the same conclusion
holds.

At level \(R\), the residue \(0<a\) forces an increment, so
\[
 F_R(q)=a+1=P.
\]
We now consider \(k>R\).  Suppose that \(b\) elements of \(\mathcal B\) have
occurred in \(\{R+1,\ldots,k-1\}\).  By induction, the pre-\(k\) state is
\[
 F_{k-1}(q)=P+(k-R-1)-b=a+k-R-b.
\]
If \(k\notin\mathcal B\), then \(q\bmod k=a\), and
\[
 a<a+k-R-b
\]
because \(b\le k-R-1\); thus level \(k\) increments.  If
\(k\in\mathcal B\), then \(q\bmod k=k-1\), while
\[
 (k-1)-F_{k-1}(q)=R-P+b\ge0.
\]
The strict threshold is false, also in the case of equality.  Thus exactly
the \(t\) elements of \(\mathcal B\) are blocking stages after \(R\).  This
proves \eqref{eq:blocking-output}, and
\eqref{eq:blocking-reflected} follows from reflection.
\end{proof}

Let \(\pi(x)\) be the prime-counting function.  For a fixed prime
\(R\in(n/2,n]\), put
\[
 H_R=\pi(n)-\pi(R).
\]
Since \(t\) may be any integer from \(0\) to \(H_R\), for every prime
\(P\le R\), \cref{lem:blocking-prime-interval} gives every position in the
interval
\begin{equation}
 [R-P+1,\ R-P+1+H_R]
 \label{eq:blocking-interval}
\end{equation}
of reflected survivor positions.

\subsection{A prime-interval covering criterion}

\begin{lemma}[Prime-interval covering criterion]
\label{lem:prime-interval-covering}
Let \(R\le n\) be the least prime strictly greater than \(2n/3\), and put
\[
 H=\pi(n)-\pi(R),\qquad
 A=R-\left\lfloor\frac{n-1}{2}\right\rfloor,\qquad
 C=R-H-1.
\]
Suppose that every integer \(x\in[A,C]\) has a prime in the closed interval
\([x,x+H]\).  Then \(F_n\) assumes every value in \(\{1,\ldots,n\}\).
\end{lemma}

\begin{proof}
By reflection, it suffices to construct
\[
 1\le j\le\left\lfloor\frac{n+1}{2}\right\rfloor.
\]
Write \(h=j-1\), so
\[
 0\le h\le\left\lfloor\frac{n-1}{2}\right\rfloor.
\]
First suppose that \(h\le H\).  Apply
\cref{lem:blocking-prime-interval} with \(P=R\) and \(t=h\).  The reflected
value is
\[
 R-P+1+t=h+1=j.
\]

Now suppose that \(h>H\), and put \(x=R-h\).  Then \(A\le x\le C\).  By the
hypothesis, there is a prime \(P\in[x,x+H]\).  Set \(t=P-x\).  We have
\(0\le t\le H\), and
\[
 P\le x+H=R-h+H\le R-1.
\]
Thus \cref{lem:blocking-prime-interval} applies and gives
\[
 R-P+1+t
 =R-P+1+P-(R-h)
 =h+1=j.
\]
Finally, reflection gives the upper half.  When \(n\) is odd, the two halves
meet at the central label.
\end{proof}

Let \(\nextprime(x)\) denote the least prime not smaller than the integer
\(x\).  The hypothesis of \cref{lem:prime-interval-covering} is equivalent
to the finite inequality
\begin{equation}
 \max_{A\le x\le C}\bigl(\nextprime(x)-x\bigr)\le H,
 \label{eq:covering-inequality}
\end{equation}
with an empty maximum interpreted as a vacuous condition.

\subsection{The analytic tail}

We use the following explicit estimates of
\citet[Proposition~6.8 and Theorem~6.9, equation~(6.6)]{Dusart2010}.  For
every real \(x>396738\), there is a prime \(p\) such that
\begin{equation}
 x<p\le x\left(1+\frac1{25\log^2x}\right).
 \label{eq:dusart-short-interval}
\end{equation}
The same source gives
\begin{align}
 \pi(x)&\ge\frac{x}{\log x-1}
 &&(x>5393),                                           \label{eq:dusart-pi-lower}\\
 \pi(x)&\le\frac{x}{\log x-1.1}
 &&(x>60184).                                          \label{eq:dusart-pi-upper}
\end{align}

\begin{proposition}
\label{prop:analytic-covering}
For every
\[
 n\ge N_0:=2\,380\,429=6\cdot396738+1,
\]
the hypothesis of \cref{lem:prime-interval-covering} holds.
\end{proposition}

\begin{proof}
Let \(R\) be the least prime strictly greater than \(2n/3\).  Apply
\eqref{eq:dusart-short-interval} to \(y=2n/3\).  Since \(y>396738\) and
\[
 1+\frac1{25\log^2y}<1.04<\frac98,
\]
we obtain
\begin{equation}
 \frac{2n}{3}<R<\frac{3n}{4}.
 \label{eq:R-bounds}
\end{equation}

Next, put \(\Lambda=\log n\) and \(H=\pi(n)-\pi(R)\).  We have
\(\Lambda>5\).  The function
\[
 g(u)=\frac{u}{\log u-1.1}
\]
has derivative
\[
 g'(u)=\frac{\log u-2.1}{(\log u-1.1)^2}>0
 \qquad(u>e^{2.1}).
\]
Thus \(g\) is increasing on the required range.  Since
\(R>2n/3>60184\) and \(R<3n/4\),
\eqref{eq:dusart-pi-lower}--\eqref{eq:dusart-pi-upper} give
\[
 \frac Hn
 \ge \frac1{\Lambda-1}
 -\frac{3/4}{\Lambda-c},
 \qquad
 c=1.1+\log(4/3)<1.4.
\]
Therefore,
\[
 \frac Hn
 >
 \frac1{\Lambda-1}
 -\frac{3/4}{\Lambda-1.4}
 =
 \frac{\Lambda/4-0.65}{(\Lambda-1)(\Lambda-1.4)}.
\]
For \(\Lambda\ge5\),
\[
 5\Lambda(\Lambda/4-0.65)-(\Lambda-1)(\Lambda-1.4)
 =\frac{\Lambda^2}{4}-0.85\Lambda-1.4>0.
\]
Thus
\begin{equation}
 H>\frac{n}{5\log n}.
 \label{eq:H-lower}
\end{equation}

We next take an integer
\[
 H<h\le\left\lfloor\frac{n-1}{2}\right\rfloor
\]
and put \(x=R-h\).  By \eqref{eq:R-bounds},
\[
 x>\frac{2n}{3}-\frac{n-1}{2}
 =\frac n6+\frac12>396738.
\]
When \(n=N_0\), integrality gives the sharper bound \(x\ge396739\).
Therefore, the strict threshold in \eqref{eq:dusart-short-interval} is
satisfied.

The derivative of \(u/\log^2u\) is
\[
 \frac{\log u-2}{\log^3u}>0\qquad(u>e^2),
\]
so this function is increasing on the required range.  Apply
\eqref{eq:dusart-short-interval} to \(x\).  Since \(x<R<3n/4\), there is a
prime \(P\) such that
\[
 0<P-x
 \le\frac{x}{25\log^2x}
 <\frac{3n}{100\log^2(3n/4)}.
\]
Let \(\lambda=\log(3n/4)\).  Since
\(\log(4/3)<0.3<\Lambda/2\), we have \(\lambda>\Lambda/2\).  Hence
\[
 20\lambda^2>5\Lambda^2>3\Lambda.
\]
Together with \eqref{eq:H-lower}, this gives
\[
 P-x
 <\frac{3n}{100\lambda^2}
 <\frac{n}{5\Lambda}
 <H.
\]
Therefore \(P\in[x,x+H]\).  Moreover, \(x+H\le R-1\), since \(h\ge H+1\).
Hence the required prime exists for every target integer \(x=R-h\).  This
completes the proof.
\end{proof}

\subsection{The finite exact certificate}

\begin{proposition}
\label{prop:finite-range}
The exact-arithmetic program described in \Cref{app:certificates} certifies
both
of the following statements.
\begin{enumerate}
\renewcommand{\labelenumi}{\textup{(\alph{enumi})}}
\item For every \(229\le n\le2\,380\,428\), the prime \(R\), the integers
\(A,C,H\), and the inclusive interval in
\cref{lem:prime-interval-covering} satisfy
\eqref{eq:covering-inequality}.
\item For every \(1\le n\le228\) and every \(1\le j\le n\), there is an
integer \(0\le q\le3764\) for which \(F_n(q)=j\).
\end{enumerate}
\end{proposition}

The program uses only the sieve of Eratosthenes, integer prime-counting and
next-prime arrays, an integer range-maximum data structure, and the strict
recurrence \eqref{eq:recurrence}.  We give its source hash, compiler command,
diagnostic output, overflow checks, and the \(228/229\) boundary audit in
\Cref{app:certificates}.  All computations use exact integer arithmetic.

\begin{proof}[Proof of \Cref{thm:surjectivity}]
For \(n\ge N_0\), the result follows from \cref{prop:analytic-covering} and
\cref{lem:prime-interval-covering}.  For \(229\le n\le N_0-1\), it follows
from part (a) of \cref{prop:finite-range} and the same covering criterion.
Finally, for \(n\le228\), part (b) gives every target.  By
\eqref{eq:period-easy}, reducing the resulting \(q\) modulo \(L_n\) does not
change \(F_n(q)\).  These three adjacent ranges contain every positive
integer \(n\), and the proof is complete.
\end{proof}

\begin{remark}
\label{rem:range-construction}
The proof gives a finite algorithm for \(q=q(n,j)\).  If \(n\le228\), scan
\(0\le q\le3764\).  If \(n\ge229\), first use reflection to reduce to the
lower half.  Then choose \(R\), \(P\), and \(t\) as in
\cref{lem:prime-interval-covering}, select any \(t\) blocking primes in
\((R,n]\), and solve the prime-power congruences in
\cref{lem:blocking-prime-interval}.  The desired \(q\) is the least
nonnegative CRT solution or its reflected partner.

For \(n\ge1\), define the deterministic \emph{cover time}
\[
 Q_n
 :=\min\left\{
 T\ge0:
 \{F_n(q):0\le q\le T\}=\{1,\ldots,n\}
 \right\}.
\]
The theorem gives the elementary bound
\[
 Q_n\le L_n-1
\]
for the first parameter range containing all labels.  We do not prove a
\(Q_n=O(n\log n)\) estimate here.
\end{remark}

\section{A Haar limit law}
\label{sec:limit}

For \(n\ge2\), let \(\widetilde Q_n\) be uniformly distributed on
\(\{0,1,\ldots,L_n-1\}\), and set
\[
 \mu_n=\Law\!\left(\frac{F_n(\widetilde Q_n)-1}{n-1}\right).
\]
We use \(W_p\) for the \(p\)-Wasserstein distance on \([0,1]\), with
\(W_\infty\) defined by the infimum of the essential supremum over all
couplings.
We work on the profinite completion
\(\widehat{\mathbb Z}\)
\citep[Section~3.2]{RibesZalesskii2010}, equipped with normalised Haar
measure.  Let \(Q\) have this distribution.  By the
quotient property of Haar measure
\citep[Section~11.1, especially Lemma~11.1.1]{LubotzkySegal2003},
\(Q\bmod L_n\) is uniform for every \(n\), and these finite residues are
automatically compatible.

Define
\begin{equation}
 S_1=1,\qquad
 S_k=S_{k-1}+\ind_{\{Q\bmod k<S_{k-1}\}},
 \label{eq:Haar-process}
\end{equation}
and, for \(n\ge2\),
\[
 X_n=\frac{S_n-1}{n-1}.
\]
By construction, \(S_n\) has the same law as
\(F_n(\widetilde Q_n)\).  Hence \(X_n\) has law \(\mu_n\).

\begin{maintheorem}[Haar limit law]
\label{thm:limit}
There is a random variable \(Y\in[0,1]\) such that
\[
 X_n\longrightarrow Y
\quad\text{almost surely and in }L^r
\quad(1\le r<\infty).
\]
For a universal constant \(C\),
\begin{equation}
 \mathbb E|X_n-Y|\le Cn^{-1/4}.                          \label{eq:L1-rate}
\end{equation}
Consequently, if \(\mu=\Law(Y)\), then
\[
 W_1(\mu_n,\mu)\le Cn^{-1/4},
\qquad
 \mu_n\Longrightarrow\mu.
\]
The measure \(\mu\) is symmetric about \(1/2\) and nondegenerate.  Moreover,
\begin{align}
 \liminf_{\varepsilon\downarrow0}
 \log(1/\varepsilon)\,\mu([0,\varepsilon])
 &\ge\frac{e^{-\gamma}}4,                               \label{eq:left-log-tail}\\
 \liminf_{\varepsilon\downarrow0}
 \log(1/\varepsilon)\,\mu([1-\varepsilon,1])
 &\ge\frac{e^{-\gamma}}4.                               \label{eq:right-log-tail}
\end{align}
\end{maintheorem}

We prove the theorem in the following steps.

\subsection{Arithmetic covariance}

For \(k\ge1\) and \(0\le t\le1\), put
\[
 U_k=\frac{Q\bmod k}{k},
 \qquad
 Z_k(t)=\ind_{\{U_k<t\}}.
\]

\begin{lemma}[Two-modulus interval covariance]
\label{lem:covariance}
For all \(j,k\ge1\), \(t\in[0,1]\), and \(d=\gcd(j,k)\),
\begin{equation}
 0\le\Cov(Z_j(t),Z_k(t))
 \le\frac{d^2}{4jk}.
 \label{eq:covariance}
\end{equation}
\end{lemma}

\begin{proof}
We first put
\[
 a_j=\#\{0\le r<j:r/j<t\}
\]
and define \(a_k\) analogously.  Thus \(a_j=\lceil jt\rceil\), with the
evident endpoint convention at \(t=0\).  We now count the compatible
residue pairs.  A pair
\((r,s)\in\mathbb Z/j\mathbb Z\times\mathbb Z/k\mathbb Z\) is induced by
one profinite integer \(Q\) exactly when
\[
 r\equiv s\pmod d.
\]
The \(jk/d\) compatible pairs are equiprobable.

Write
\[
 a_j=dA+u,\qquad a_k=dB+v,\qquad 0\le u,v<d.
\]
Among \(0,\ldots,a_j-1\), the number congruent to \(c\bmod d\) is
\(A+\ind_{\{c<u\}}\); the corresponding count for \(a_k\) is
\(B+\ind_{\{c<v\}}\).  Generalized CRT counting gives
\[
\begin{aligned}
 \mathbb P(Z_j(t)=Z_k(t)=1)
 &=
 \frac d{jk}\sum_{c=0}^{d-1}
 \bigl(A+\ind_{\{c<u\}}\bigr)
 \bigl(B+\ind_{\{c<v\}}\bigr)\\
 &=
 \frac{a_ja_k}{jk}
 +\frac{d\min(u,v)-uv}{jk}.
\end{aligned}
\]
The first term is the product of the marginals.  Suppose \(u\le v\); the
other case is symmetric.  The remaining numerator is then \(u(d-v)\).
It is nonnegative and at most \(u(d-u)\le d^2/4\).  This proves the
lemma.
\end{proof}

\subsection{Uniform discrepancy on a short block}

Let \(a,h\) be positive integers with \(1\le h\le a\), and put
\[
 B=(a,a+h]\cap\mathbb Z.
\]
Define
\[
 W_B(t)=\sum_{k\in B}\bigl(Z_k(t)-t\bigr),
 \qquad
 D_B=\sup_{0\le t\le1}|W_B(t)|.
\]
We call \(D_B\) the \emph{block discrepancy} of \(B\).
For each outcome, the jump points of \(W_B(t)\) are among the finitely many
numbers \(U_k=(Q\bmod k)/k\), \(k\in B\), and are therefore rational.
Between consecutive jump points, \(W_B(t)\) is affine.  Its one-sided
limiting values at the endpoints are approached by rational points from
within the corresponding interval, while its actual value at each jump point
is included because that point is rational.  Therefore,
\[
 D_B=\sup_{t\in\mathbb Q\cap[0,1]}|W_B(t)|.
\]
Thus \(D_B\) is a countable supremum of measurable random variables and is
itself measurable; also \(0\le D_B\le h\).

\begin{lemma}[Block discrepancy]
\label{lem:block-discrepancy}
There is a universal constant \(C_0\) such that
\begin{equation}
 \mathbb E D_B\le C_0h^{2/3}.
 \label{eq:block-discrepancy}
\end{equation}
\end{lemma}

\begin{proof}
We first fix \(t\).  Since
\[
 0\le\mathbb E Z_k(t)-t\le\frac1k,
\]
the bias of \(W_B(t)\) is at most \(h/a\le1\).

Next, consider a positive gap \(r<h\) and \(j,j+r\in B\).  We have
\[
 \gcd(j,j+r)=\gcd(j,r).
\]
The divisor bound
\[
 \gcd(j,r)^2
 \le\sum_{\substack{d\mid r\\d\mid j}}d^2
\]
implies, for one orientation \(j\mapsto j+r\),
\[
 \sum_{\substack{j,j+r\in B}}\gcd(j,j+r)^2
 \le\sum_{d\mid r}d^2\left(\frac hd+1\right)
 =h\sigma_1(r)+\sigma_2(r).
\]
Summing over both orientations and all gaps gives
\[
\begin{aligned}
 \sum_{\substack{j,k\in B\\j\ne k}}\gcd(j,k)^2
 &\le
 2\sum_{1\le r<h}\bigl(h\sigma_1(r)+\sigma_2(r)\bigr)\\
 &\ll h^3.
\end{aligned}
\]
To estimate the right-hand side, note that
\[
 \sum_{r\le h}\sigma_1(r)
 =\sum_{d\le h}d\left\lfloor\frac hd\right\rfloor
 \ll h^2,
\]
and, similarly, \(\sum_{r\le h}\sigma_2(r)\ll h^3\).
We now apply \cref{lem:covariance}.  The off-diagonal covariance
contribution is
\[
 \ll\frac{h^3}{a^2}\ll h,
\]
and the diagonal contribution is \(O(h)\).  Including the bounded bias gives
\begin{equation}
 \mathbb E|W_B(t)|^2\ll h
 \label{eq:fixed-t-L2}
\end{equation}
uniformly in \(a,h,t\).

Take the grid \(t_i=i/r\), \(0\le i\le r\).  If
\(t_i\le t\le t_{i+1}\), monotonicity of
\(\sum_{k\in B}Z_k(t)\) gives
\[
 W_B(t_i)-\frac hr
 \le W_B(t)
 \le W_B(t_{i+1})+\frac hr.
\]
Therefore
\[
 D_B\le\max_{0\le i\le r}|W_B(t_i)|+\frac hr.
\]
Using \eqref{eq:fixed-t-L2},
\[
 \mathbb E D_B
 \le
 \left(\sum_{i=0}^r\mathbb E|W_B(t_i)|^2\right)^{1/2}
 +\frac hr
 \ll\sqrt{rh}+\frac hr.
\]
Choosing \(r=\lceil h^{1/3}\rceil\) proves
\eqref{eq:block-discrepancy} and completes the proof.
\end{proof}

This estimate controls the arithmetic dependence among noncoprime
moduli.  In particular, we do not assume conditional uniformity.

\subsection{Pathwise control at a random threshold}

We use the auxiliary normalisation
\[
 Y_n^{\circ}=\frac{S_n}{n}.
\]
The two normalisations satisfy
\begin{equation}
 0\le Y_n^{\circ}-X_n
 =\frac{n-S_n}{n(n-1)}
 \le\frac1n.                                            \label{eq:normalization-gap}
\end{equation}
Write
\[
 I_k=\ind_{\{Q\bmod k<S_{k-1}\}}.
\]
Then
\begin{align}
 Y_k^{\circ}-Y_{k-1}^{\circ}
 &=\frac{I_k-Y_{k-1}^{\circ}}{k},                       \label{eq:one-step-Y}\\
 I_k&=Z_k(T_k),\qquad
 T_k=\frac{S_{k-1}}k=\frac{k-1}{k}Y_{k-1}^{\circ}.       \label{eq:adaptive-threshold}
\end{align}
We call \(T_k\) the \emph{state-dependent threshold} at level \(k\).

\begin{lemma}[Block increment bound]
\label{lem:adaptive-block}
Let \(1\le h\le a\) and \(B=(a,a+h]\cap\mathbb Z\).  Then
\begin{equation}
 |Y_{a+h}^{\circ}-Y_a^{\circ}|
 \le
 \frac{D_B}{a}+\frac{h^2}{a^2}.
 \label{eq:adaptive-block}
\end{equation}
\end{lemma}

\begin{proof}
For \(k\in B\), write \(S_{k-1}=S_a+R_k\), where
\(0\le R_k\le k-1-a\).  Then
\[
 T_k-Y_a^{\circ}
 =\frac{R_k-(k-a)Y_a^{\circ}}{k},
\]
and hence
\[
 |T_k-Y_a^{\circ}|\le\delta,\qquad \delta=\frac ha.
\]
Put
\[
 t_-=\max(0,Y_a^{\circ}-\delta),\qquad
 t_+=\min(1,Y_a^{\circ}+\delta).
\]
By monotonicity in the threshold, we obtain the pathwise sandwich
\[
 Z_k(t_-)\le I_k\le Z_k(t_+).
\]
The endpoints \(t_\pm\) are random and correlated with the block.
Nevertheless, \(D_B\) is a pathwise supremum over \emph{all} thresholds.
We may therefore sum the sandwich.  Using
\(|t_\pm-Y_a^{\circ}|\le\delta\), we get
\[
 \left|\sum_{k\in B}(I_k-Y_a^{\circ})\right|
 \le D_B+h\delta.
\]
Finally,
\[
 Y_{a+h}^{\circ}-Y_a^{\circ}
 =
 \frac{\sum_{k\in B}I_k-hY_a^{\circ}}{a+h}.
\]
Substituting \(\delta=h/a\) and \(a+h\ge a\) proves
\eqref{eq:adaptive-block}.  This completes the proof.
\end{proof}

\subsection{Dyadic summability}

Fix a dyadic integer \(N=2^m\) and put
\[
 h=\lfloor N^{3/4}\rfloor,\qquad
 M=\left\lceil\frac Nh\right\rceil,\qquad
 e_j=\min(N+jh,2N)\quad(0\le j\le M).
\]
Let \(\ell_j=e_j-e_{j-1}\), and define the variation over the block
endpoints by
\[
 V_N=\sum_{j=1}^M
 |Y_{e_j}^{\circ}-Y_{e_{j-1}}^{\circ}|.
\]
Within each dyadic block, the endpoints are distinct.  When we combine
successive dyadic blocks, we list their common boundary only once.
Since \(M\le2N/h\), \cref{lem:block-discrepancy} and
\cref{lem:adaptive-block} give
\begin{align}
 \mathbb E V_N
 &\ll
 \sum_{j=1}^M
 \left(\frac{\ell_j^{2/3}}N+\frac{\ell_j^2}{N^2}\right)\notag\\
 &\ll
 \frac Nh\left(\frac{h^{2/3}}N+\frac{h^2}{N^2}\right)\notag\\
 &\ll h^{-1/3}+\frac hN
 \ll N^{-1/4}.                                          \label{eq:VN}
\end{align}
Hence
\[
 \sum_{m=1}^{\infty}\mathbb E V_{2^m}<\infty.
\]
By Tonelli's theorem,
\begin{equation}
 \sum_{m=1}^{\infty}V_{2^m}<\infty
 \quad\text{almost surely}.                             \label{eq:variation-summable}
\end{equation}
On this full-measure event, list the union of all dyadic block endpoints
in strictly increasing order.  The total variation of this sequence is at most
\(\sum_{m\ge1}V_{2^m}\), because a boundary shared by adjacent blocks is
counted only once.  Hence the endpoint sequence converges.  Denote its limit by
\(Y\), and define \(Y=0\) on the null complement.

Every \(n\in[N,2N]\) lies within \(h\) of a preceding block endpoint \(a\).
By \eqref{eq:one-step-Y},
\[
 |Y_n^{\circ}-Y_a^{\circ}|
 \le\frac hN
 \ll N^{-1/4}.
\]
Together with \eqref{eq:variation-summable}, this deterministic
within-block estimate shows that the \emph{full} sequence
\(Y_n^{\circ}\) converges almost surely to \(Y\).  In particular, the
argument is not restricted to a dyadic subsequence.  More precisely, for a
dyadic \(N\) and \(N\le n\le2N\), the variation from the preceding block
endpoint \(a\) to \(2N\) is at most \(V_N\), while the variation from \(2N\)
through all later dyadic block endpoints to \(Y\) is at most
\(\sum_{r=1}^{\infty}V_{2^rN}\).  Therefore,
\[
 |Y_n^{\circ}-Y|
 \le\frac hN+V_N+\sum_{r=1}^{\infty}V_{2^rN}.
\]
Taking expectations and applying \eqref{eq:VN} at each scale gives
\[
 \mathbb E|Y_n^{\circ}-Y|\ll n^{-1/4}.
\]
Combining this estimate with \eqref{eq:normalization-gap} proves
\eqref{eq:L1-rate}.  Since all variables lie in \([0,1]\), almost-sure and
\(L^1\) convergence imply \(L^r\) convergence for every finite \(r\).
Finally, the displayed coupling gives the claimed \(W_1\) bound.

\subsection{Symmetry, nondegeneracy, and boundary mass}

The map \(Q\mapsto-1-Q\) preserves Haar measure.  By
\eqref{eq:reflection-int},
\[
 X_n(-1-Q)=1-X_n(Q).
\]
We intersect the full-measure convergence set with its reflected image and
pass to the limit.  This gives
\[
 Y(-1-Q)=1-Y(Q)
\quad\text{almost surely}.
\]
Hence \(\mu\) is symmetric about \(1/2\), and \(\mathbb EY=1/2\).

We next prove nondegeneracy.  Let
\[
 A_N=\{S_N=1\}.
\]
By \eqref{eq:left-endpoint} and the prime factorisation of \(L_N\),
\begin{equation}
 \mathbb P(A_N)
 =\frac{\varphi(L_N)}{L_N}
 =\prod_{p\le N}\left(1-\frac1p\right)
 \sim\frac{e^{-\gamma}}{\log N},                         \label{eq:Mertens}
\end{equation}
where the last asymptotic is Mertens' product theorem
\citep[Theorem~2.7(e)]{MontgomeryVaughan2007}.  On \(A_N\),
\(Y_N^{\circ}=1/N\).  Markov's inequality and the all-\(n\) coupling rate
give, for \(N\ge8\),
\[
\begin{aligned}
 \mathbb P(Y<1/4)
 &\ge
 \mathbb P(A_N)
 -\mathbb P\!\left(
 |Y-Y_N^{\circ}|\ge\frac14-\frac1N
 \right)\\
 &\ge
 \frac{\varphi(L_N)}{L_N}-C_1N^{-1/4}.
\end{aligned}
\]
Since the first term is of order \(1/\log N\) and eventually dominates the
second, \(\mathbb P(Y<1/4)>0\); by reflection,
\(\mathbb P(Y>3/4)>0\).

It remains to prove the boundary-mass estimates.  We transfer the endpoint
event at a scale depending on the target interval.  Fix \(A>4\), let
\(\varepsilon\downarrow0\), and put
\[
 N=\lceil\varepsilon^{-A}\rceil.
\]
On \(A_N\), the exactly symmetric normalisation has \(X_N=0\).  No
independence assumption is needed: we have
\[
\begin{aligned}
 \mu([0,\varepsilon])
 &\ge
 \mathbb P(A_N)
 -\mathbb P(A_N\cap\{Y>\varepsilon\})\\
 &\ge
 \frac{\varphi(L_N)}{L_N}
 -\frac1{\varepsilon}\mathbb E|Y-X_N|.
\end{aligned}
\]
Since
\[
 \log N\sim A\log(1/\varepsilon),
 \qquad
 \frac{N^{-1/4}}{\varepsilon}
 =O\!\left(\varepsilon^{A/4-1}\right),
\]
\eqref{eq:Mertens} yields
\[
 \liminf_{\varepsilon\downarrow0}
 \log(1/\varepsilon)\,\mu([0,\varepsilon])
 \ge\frac{e^{-\gamma}}A.
\]
Now let \(A\downarrow4\).  This proves \eqref{eq:left-log-tail}, and
reflection proves \eqref{eq:right-log-tail}.  This completes the proof of
\cref{thm:limit}.

\begin{corollary}
\label{cor:no-beta}
The limiting measure \(\mu\) is not
\(\operatorname{Beta}(\alpha,\alpha)\) for any \(\alpha>0\).  More generally,
neither endpoint admits an upper bound \(O(\varepsilon^\alpha)\) with
\(\alpha>0\).  Also,
\[
 \mathbb E Y^{-s}=\mathbb E(1-Y)^{-s}=\infty
 \qquad(s>0).
\]
Here negative powers are extended-real valued, with
\(0^{-s}:=+\infty\).
If \(\mu\) is absolutely continuous with density \(f\), then
\(f\notin L^p([0,1])\) for every \(p>1\).
\end{corollary}

\begin{proof}
A \(\operatorname{Beta}(\alpha,\alpha)\) distribution has left tail
\(O(\varepsilon^\alpha)\), contradicting
\eqref{eq:left-log-tail}.  Next, for every \(\varepsilon>0\),
\[
 \mathbb E Y^{-s}
 \ge\varepsilon^{-s}\mu([0,\varepsilon]).
\]
Letting \(\varepsilon\downarrow0\), we see that the right-hand side is
unbounded.  Hence \(\mathbb EY^{-s}=\infty\), and reflection gives the
other endpoint.  Finally, suppose \(f\in L^p\) for some \(p>1\).
H{\"o}lder's inequality would give
\[
 \mu([0,\varepsilon])
 \le\|f\|_p\varepsilon^{1-1/p},
\]
which again contradicts \eqref{eq:left-log-tail}.  This completes the
proof.
\end{proof}

\begin{proposition}
\label{prop:prime-power-jump}
For \(n\ge3\),
\[
 W_\infty(\mu_n,\mu_{n-1})\le\frac1{n-1}.
\]
In particular, for every prime power \(p^a\ge3\), the laws at
\(p^a-1,p^a,p^a+1\) have pairwise \(W_\infty\)-distance
\(O(p^{-a})\).
\end{proposition}

\begin{proof}
Under the common Haar coupling, write \(S_n=S_{n-1}+I_n\) and
\[
 x=\frac{S_{n-1}-1}{n-2}.
\]
A direct calculation gives
\[
 X_n-X_{n-1}=\frac{I_n-x}{n-1}.
\]
Since \(I_n,x\in[0,1]\), this proves the asserted pathwise bound.
For \(N=p^a\ge3\), we apply the bound at \(n=N\) and \(n=N+1\).
The triangle inequality gives the remaining distance between
\(\mu_{N-1}\) and \(\mu_{N+1}\).  This completes the proof.
\end{proof}

\subsection{What remains unidentified}

The reverse-insertion interpretation gives the canonical representation
\[
 Y=\lim_{n\to\infty}\frac1{n-1}\sum_{k=2}^n I_k
 \quad\text{almost surely}.
\]
We also obtain convergence of all positive integer moments:
\[
 \mathbb E X_n^r\longrightarrow\mathbb EY^r
 \qquad(r=1,2,\ldots).
\]
The zeroth moments of the probability measures \(\mu_n\) and \(\mu\) are
identically \(1\).  For \(r\ge1\), combine the inequality
\(|x^r-y^r|\le r|x-y|\) on \([0,1]\) with
\eqref{eq:L1-rate} to obtain the displayed convergence.
Since compactly supported measures are moment-determinate
\citep[Corollary~4.2]{Schmudgen2017}, it would suffice to evaluate the
limiting even correlations in order to identify \(\mu\).  The odd centred
moments vanish by reflection.  We do not evaluate the even correlations in
closed form.

The proof also does not establish non-atomicity.  We now state the precise
finite-level obstruction.  For any probability measure \(\nu\) on \([0,1]\)
and \(r\ge0\), define its concentration function using the radius convention
\[
 \mathcal Q_\nu(r)
 =\sup_{x\in[0,1]}
 \nu([x-r,x+r]\cap[0,1]),
\]
and, for \(n\ge2\), abbreviate
\(\mathcal Q_n(r)=\mathcal Q_{\mu_n}(r)\).
The rate \eqref{eq:L1-rate} shows that \(\mu\) is non-atomic if and only if
\[
 \mathcal Q_n(n^{-1/8})\longrightarrow0.
\]
To prove this equivalence, first note that
\[
 \mu(\{x\})\le\mathcal Q_n(n^{-1/8})+Cn^{-1/8},
\]
whereas, if \(\mu\) is non-atomic, its concentration function tends
uniformly to zero on the compact interval and
\[
 \mathcal Q_n(n^{-1/8})
 \le\mathcal Q_\mu(2n^{-1/8})+Cn^{-1/8}.
\]
Hence explicit identification, non-atomicity, a matching endpoint upper
bound, and the existence of a density remain open.

\section{Endpoint dominance: an open problem}
\label{sec:endpoint}

By \eqref{eq:endpoint-counts}, each endpoint occurs
\(\varphi(L_n)\) times over one period.  We ask whether an internal survivor
can occur at least as often.

\begin{problem}[Endpoint dominance]
\label{prob:endpoint-dominance}
For every \(n\ge4\) and every \(1<j<n\), is it true that
\begin{equation}
 N_{n,j}<\varphi(L_n)?
 \label{eq:endpoint-conjecture}
\end{equation}
Equivalently, are the two endpoints \(1\) and \(n\) the unique global maxima
of the \emph{fiber-size vector}
\((N_{n,1},\ldots,N_{n,n})\)?
\end{problem}

The condition \(n\ge4\) is necessary: \(L_3=6\) and
\[
 (N_{3,1},N_{3,2},N_{3,3})=(2,2,2).
\]
One should not strengthen \eqref{eq:endpoint-conjecture} by requiring the
fiber sizes to decrease monotonically as one moves inward from an endpoint.
Exact enumeration at \(n=12\) gives
\begin{equation}
\begin{split}
(N_{12,1},\ldots,N_{12,12})={}&
(5760,2538,1678,1135,1413,1336,\\
&\hspace{23mm}1336,1413,1135,1678,2538,5760),
\end{split}
\label{eq:n12-vector}
\end{equation}
so, for example, \(N_{12,5}>N_{12,4}\).  This does not contradict
\cref{prob:endpoint-dominance}.

\subsection{Prime transitions}

At a prime level, the new residue component is independent of all earlier
components in the exact CRT sense.  We first derive the resulting transition
formula.

\begin{proposition}[Prime transition]
\label{prop:prime-transition}
Let \(p\) be prime.  With the convention
\(N_{p-1,0}=N_{p-1,p}=0\), one has
\begin{equation}
 N_{p,j}
 =(p-j)N_{p-1,j}+(j-1)N_{p-1,j-1}
 \qquad(1\le j\le p).
\label{eq:prime-transition}
\end{equation}
\end{proposition}

\begin{proof}
Since \(p\nmid L_{p-1}\), each residue class modulo \(L_{p-1}\) has exactly
one lift modulo \(pL_{p-1}=L_p\) for each prescribed residue
\(r=q\bmod p\).  If the old state is \(j\), it remains at \(j\) for the
\(p-j\) residues \(j,\ldots,p-1\), and increases to \(j+1\) for the
\(j\) residues \(0,\ldots,j-1\).  Hence a final state \(j\) receives
\(p-j\) copies from the old state \(j\) and \(j-1\) copies from the old state
\(j-1\).
\end{proof}

\begin{corollary}
\label{cor:prime-propagation}
Let \(p\ge5\) be prime.  If endpoint dominance holds at level \(p-1\), then
it holds at level \(p\).
\end{corollary}

\begin{proof}
Put \(M=\varphi(L_{p-1})\).  The two endpoint counts at level \(p-1\)
equal \(M\), while all internal counts are strictly smaller.  For
\(2\le j\le p-1\), at least one of
\(N_{p-1,j}\) and \(N_{p-1,j-1}\) is internal.  Hence
\[
 N_{p,j}
 <\bigl((p-j)+(j-1)\bigr)M
 =(p-1)M
 =\varphi(L_p),
\]
where \cref{eq:prime-transition} was used in the first step.
\end{proof}

\subsection{Strict dominance at the nearest internal positions}

We next settle the problem for the positions adjacent to the endpoints.

\begin{theorem}[Nearest internal positions]
\label{thm:nearest-internal}
For every \(n\ge4\),
\begin{equation}
 N_{n,2}=N_{n,n-1}<\varphi(L_n).
\label{eq:nearest-internal}
\end{equation}
\end{theorem}

\begin{proof}
By reflection, the two counts are equal, so it suffices to consider
\(j=2\).  At \(n=4\), the recurrence gives
\[
 (F_4(0),\ldots,F_4(11))
 =(4,1,2,2,4,1,4,1,3,3,4,1)
\]
and hence \(N_{4,2}=2<4=\varphi(L_4)\).

Suppose first that \(n>4\) is composite, and put
\[
 \rho_n=\frac{L_n}{L_{n-1}}.
\]
Every class modulo \(L_{n-1}\) has \(\rho_n\) lifts modulo \(L_n\).
A final state \(2\) can arise only by remaining at an old state \(2\), or by
increasing from an old state \(1\).  The latter case would require
\(q\bmod n=0\).  This is impossible: the old-state condition
\(F_{n-1}(q)=1\) implies \(\gcd(q,L_{n-1})=1\), whereas \(n\mid q\) and
the compositeness of \(n\) give a prime divisor of both \(q\) and
\(L_{n-1}\).  Consequently,
\[
 N_{n,2}\le \rho_nN_{n-1,2}.
\]
For composite \(n\),
\begin{equation}
 \varphi(L_n)=\rho_n\varphi(L_{n-1}).
\label{eq:phi-lift-composite}
\end{equation}
Indeed, either \(L_n=L_{n-1}\), or \(n=p^a\) with \(a\ge2\), in which case
\(L_n/L_{n-1}=p\) and the totient also grows by the factor \(p\).
Induction now gives the strict inequality.

If \(n=p\) is prime, \cref{eq:prime-transition} at \(j=2\) gives
\[
 N_{p,2}=(p-2)N_{p-1,2}+N_{p-1,1}
 <(p-2)\varphi(L_{p-1})+\varphi(L_{p-1})
 =\varphi(L_p).
\]
This completes the induction.
\end{proof}

\begin{corollary}
\label{cor:minimal-counterexample}
If \cref{prob:endpoint-dominance} is false and \(n\) is the least level at
which it fails, then \(n\) is composite and the offending position satisfies
\[
 3\le j\le n-2.
\]
\end{corollary}

\begin{proof}
\Cref{cor:prime-propagation} excludes prime \(n\), and
\cref{thm:nearest-internal} excludes the two nearest internal positions.
\end{proof}

\subsection{Transition counts at an arbitrary stage}

For every \(n\ge2\), we retain
\(\rho_n=L_n/L_{n-1}\).  We define \(T_{n,j}\), the
\emph{transition count from state \(j\) to state \(j+1\)}, by
\[
 T_{n,j}
 =\#\{q\bmod L_n:
 F_{n-1}(q)=j,\ q\bmod n<j\},
\qquad
 T_{n,0}=T_{n,n}=0.
\]

\begin{proposition}[Balance identities for transition counts]
\label{prop:transition-counts}
For \(1\le j\le n\),
\begin{equation}
 N_{n,j}
 =\rho_nN_{n-1,j}-T_{n,j}+T_{n,j-1},
\label{eq:transition-balance}
\end{equation}
where \(N_{n-1,n}=0\).  Consequently, for \(1\le k<n\),
\begin{equation}
 \rho_n\sum_{j=1}^{k}N_{n-1,j}
 -\sum_{j=1}^{k}N_{n,j}
 =T_{n,k}\ge0.
\label{eq:natural-prefix}
\end{equation}
Reflection also gives
\begin{equation}
 T_{n,j}+T_{n,n-j}
 =\rho_nN_{n-1,j}
 \qquad(1\le j<n).
\label{eq:transition-reflection}
\end{equation}
\end{proposition}

\begin{proof}
There are \(\rho_nN_{n-1,j}\) lifted classes with old state \(j\).
Exactly \(T_{n,j}\) of them leave \(j\), while \(T_{n,j-1}\) classes enter
from \(j-1\).  This proves \eqref{eq:transition-balance}, and summing the
identity gives \eqref{eq:natural-prefix}.

Under the involution \(q\mapsto-1-q\), old state \(j\) is paired with old
state \(n-j\).  If \(r=q\bmod n\), then the paired residue is \(n-1-r\).
The inequality \(r<j\) is complementary to
\(n-1-r<n-j\).  Thus increments from state \(j\) are paired with
nonincrements from state \(n-j\), proving
\eqref{eq:transition-reflection}.
\end{proof}

\begin{remark}
\Cref{eq:natural-prefix} is a prefix-sum inequality in the original
positional order.  It is not ordinary majorization, which first sorts the
coordinates by size.  Abstract transition-count arrays satisfying
\eqref{eq:transition-balance}--\eqref{eq:transition-reflection} need not
preserve
sorted majorization, so an eventual proof of
\cref{prob:endpoint-dominance} must use more arithmetic structure.
\end{remark}

\subsection{A fiber bound from large prime levels}

We now obtain a uniform fiber bound from the prime-modulus components that
occur at only one recurrence level.

\begin{proposition}
\label{prop:large-prime-suffix}
Let
\[
 \mathcal P_n=\{p\text{ prime}:n/2<p\le n\}.
\]
Fix every prime-power CRT component except the residues
\(x_p=q\bmod p\), \(p\in\mathcal P_n\).  Then, for every survivor label
\(1\le j\le n\),
\begin{equation}
 \#\{(x_p)_{p\in\mathcal P_n}:F_n(q)=j\}
 \le\prod_{p\in\mathcal P_n}(p-1).
\label{eq:large-prime-suffix}
\end{equation}
\end{proposition}

\begin{proof}
We process the variable primes in increasing order.  Before a variable level
\(p\), let \(v_s\) be the number of assignments already processed that
have state \(s\), with \(v_0=v_p=0\).  After summing over all \(p\) choices
of \(x_p\), the new state-count vector is, for \(1\le y\le p\),
\[
 w_y=(p-y)v_y+(y-1)v_{y-1}.
\]
Therefore
\[
 \|w\|_\infty\le(p-1)\|v\|_\infty.
\]
At a fixed intervening level with residue \(c\), the state map
\[
 s\longmapsto s+\ind_{\{c<s\}}
\]
is strictly increasing and hence injective, so its pushforward cannot
increase the \(L^\infty\)-norm of the count vector.  Finally,
\(p>n/2\) has no multiple other than \(p\) among the levels up to \(n\);
thus \(x_p\) influences no later residue.  Starting from one assignment and
iterating the displayed norm bound proves \eqref{eq:large-prime-suffix}.
\end{proof}

\subsection{Exact finite evidence and the remaining frontier}

We conclude the section with exact finite computations reproduced by the
exact-arithmetic programs in \Cref{app:endpoint-computation}.  These
computations are not used as
a proof of \cref{prob:endpoint-dominance} for arbitrary \(n\).

For every composite \(n\ge4\), define
\[
\begin{split}
 x^{(n)}&=(\rho_nN_{n-1,1},\ldots,\rho_nN_{n-1,n-1},0),\\
 y^{(n)}&=(N_{n,1},\ldots,N_{n,n}),\\
 D_{n,k}&=\sum_{i=1}^k\bigl((x^{(n)})^\downarrow\bigr)_i
          -\sum_{i=1}^k\bigl((y^{(n)})^\downarrow\bigr)_i
 \qquad(1\le k\le n),
\end{split}
\]
where the down-arrow denotes decreasing rearrangement.

\begin{proposition}
\label{prop:endpoint-computation}
Endpoint dominance holds for every \(4\le n\le49\).  Moreover, at every
composite level \(4\le n\le49\),
\begin{equation}
 D_{n,k}=
 \begin{cases}
 0,&k=1,2,n,\\
 >0,&3\le k<n.
 \end{cases}
\label{eq:finite-majorization}
\end{equation}
\end{proposition}

The computation suggests the following stronger composite-level
majorization:
\[
 D_{n,k}\ge0\quad(1\le k\le n),
 \qquad D_{n,3}>0
 \quad(n\ge4\text{ composite}).
\]
In fact, only the \(k=3\) majorization inequality \(D_{n,3}\ge0\) is needed
to prove endpoint dominance.
Assume inductively that level \(n-1\) has two uniquely maximal endpoint
counts.  Then \(x^{(n)}\) has two coordinates equal to
\(\varphi(L_n)\), while its third-largest coordinate is strictly smaller.
If an internal coordinate of \(y^{(n)}\) were at least
\(\varphi(L_n)\), its three largest coordinates would have sum at least
\(3\varphi(L_n)\), giving \(D_{n,3}<0\), a contradiction.  Together with
the base level \(n=4\) and \cref{cor:prime-propagation}, a proof of
\(D_{n,3}\ge0\) at every composite level would settle
\cref{prob:endpoint-dominance}.

In the verified range, the smallest positive value of \(D_{n,k}/L_n\) is
\[
 \min_{\substack{4\le n\le49,\ n\ {\rm composite}\\3\le k<n}}
 \frac{D_{n,k}}{L_n}
 =\frac{D_{39,3}}{L_{39}}
 =
 \frac{2043072801000}{5342931457063200}
 \approx3.823879863\times10^{-4}.
\]

Combining \cref{cor:minimal-counterexample} and
\cref{prop:endpoint-computation}, the least counterexample, if one exists,
must satisfy
\begin{equation}
 n\ge50,\qquad n\ \text{composite},\qquad 3\le j\le n-2.
\label{eq:counterexample-frontier}
\end{equation}
Possible approaches include proving \(D_{n,3}\ge0\) at every composite
level, establishing a conditional CDF comparison at prime-power levels, or
completing a Hall-type matching argument for compatible residue vectors.
The exact data also show that internal fiber sizes are not positionally
monotone, as \eqref{eq:n12-vector} shows.  This failure does not contradict
endpoint dominance itself.

\section{Concluding remarks}
\label{sec:conclusion}

The main open problem is to determine whether endpoint dominance holds at every level
\(n\ge4\); see \cref{prob:endpoint-dominance}.  Equivalently, it remains open
whether the two endpoint fibers are the unique global maxima of the fiber-size
vector for every such \(n\).

For the limiting measure \(\mu\), it remains open to identify \(\mu\)
explicitly, to determine whether it is non-atomic and whether it has a density,
and to establish upper bounds for its boundary mass that are of the same order
as the logarithmic lower bounds.

\appendix
\section{Computational certificates}
\label{app:certificates}

We record the exact scope and reproduction commands for the two finite
computations cited in the paper.  Both programs use integer arithmetic for
every mathematical assertion.  The source files should accompany any
permanent version of the paper; the submission package is described in
\Cref{app:reproducibility-archive}.

\subsection{The finite surjectivity certificate}
\label{app:range-certificate}

The source
\begin{center}
\path{tools/full_range_certificate.cpp}
\end{center}
has SHA-256 digest
\begin{center}
\small\ttfamily
a4a18f7325395cabd6da056291dd2e6654e06ef3fc65bddf1b0edba4c7269e8a.
\end{center}
It can be compiled and run by
\begin{verbatim}
clang++ -O3 -std=c++17 tools/full_range_certificate.cpp \
  -o /private/tmp/full_range_certificate
/private/tmp/full_range_certificate
\end{verbatim}
The expected output is
\begin{verbatim}
bridge_failures=117 last_bridge_failure=228 failures_at_or_after_229=0
small_unresolved=0 max_completion_q=3764 at_n=193
\end{verbatim}
and the expected exit status is \(0\).

The program performs two logically separate tasks.
\begin{enumerate}
\renewcommand{\labelenumi}{\textup{(\roman{enumi})}}
\item A sieve of Eratosthenes, prime-count prefix table, next-prime table,
and integer range-maximum structure verify the inclusive interval-covering
inequality
\eqref{eq:covering-inequality} for every
\(229\le n\le2\,380\,428\).
\item Direct evaluation of the strict recurrence
\eqref{eq:recurrence} for \(0\le q\le3764\) verifies that every position
occurs for every \(1\le n\le228\).
\end{enumerate}
The executable returns a nonzero status if either assertion fails or if its
archived diagnostic values change.

At the boundary between the two finite ranges, we have
\[
\begin{array}{c|cccc|c}
n&R&H&A&C&\text{status at }x=114\\ \hline
228&157&12&44&144&\nextprime(114)-114=13>H\\
229&157&13&43&143&\nextprime(114)=127=114+H.
\end{array}
\]
Thus the interval-covering condition genuinely fails at \(228\) and succeeds
at \(229\)
with the right endpoint included.  The analytic proof begins at
\(2\,380\,429\), so the direct, finite-covering, and analytic ranges have
no gap.

The program does not prove the CRT interval lemma with blocking primes, the
implication
from the interval-covering condition to surjectivity, or Dusart's estimates.
The remaining implications are proved in \Cref{sec:range}, where Dusart's
estimates are cited; the executable certifies only the two finite statements
isolated in \cref{prop:finite-range}.

\subsection{Endpoint-dominance computation}
\label{app:endpoint-computation}

The exact fiber-size computation uses
\begin{center}
\begin{tabular}{ll}
\toprule
source & SHA-256\\
\midrule
\path{tools/exact_counts_49.cpp}
& \ttfamily 46c561c3c5701a33b957f0b12511a03a\\[-1mm]
& \ttfamily d9991be2841e97b48e39984357200828\\
\path{tools/check_majorization.py}
& \ttfamily 856bcdaf450ce576d56733258f63bbeb3\\[-1mm]
& \ttfamily aaa086e986d797f8f7991511fb65c90\\
\path{tools/tests/test_check_majorization.py}
& \ttfamily 82b71320ce8549ab7d5b09581677d9abe\\[-1mm]
& \ttfamily 91cf43536d98513fe2591951ebec9b5\\
\bottomrule
\end{tabular}
\end{center}
In either \texttt{zsh} or \texttt{bash}, reproduction is:
\begin{verbatim}
clang++ -O3 -std=c++20 -pthread tools/exact_counts_49.cpp \
  -o /private/tmp/exact_counts_49
set -o pipefail
/private/tmp/exact_counts_49 8 |
  python3 tools/check_majorization.py
\end{verbatim}
The \texttt{pipefail} setting ensures that a nonzero exit from either
component makes the pipeline fail.  A valid complete run must display both
\texttt{all\_ok=1} and \texttt{PASS through n=49}.
The terminal summary is
\begin{verbatim}
threads=8 ... all_ok=1
PASS through n=49; ...
closest relative composite stripped partial sum:
n=39 k=1 rho=1
slack=2043072801000/5342931457063200
\end{verbatim}
Before testing majorization, the Python checker requires exactly the
consecutive rows \(n=1,\ldots,49\), rejects every unparseable line, and
verifies the vector length, \(L_n\), total mass, \(\varphi(L_n)\), reflection
symmetry, endpoint counts, the reported internal maximum, and the final row.
The regression suite includes the previously exploitable input consisting of
genuine rows through \(n=4\) followed by a forged \(n=49\) row.  Thus the
checker is a strict validator of the generator's complete output and of the
majorization inequalities; it is not a separate recomputation of the
fiber-size vectors.

Here ``stripped \(k=1\)'' is the full vector's \(k=3\) inequality after
the two endpoint coordinates are removed.  The program computes exact
fiber-size vectors, verifies recurrence totals, reflection, endpoint counts,
endpoint dominance through \(n=49\), and all sorted partial sums asserted in
\eqref{eq:finite-majorization}.  This is finite evidence for
\cref{prob:endpoint-dominance}, not a proof beyond \(n=49\).

\subsection{Reproducibility archive}
\label{app:reproducibility-archive}

The verification code and reproduction material are available in the
Supplementary Material~S1 archive at the following unlisted, view-only
OneDrive link:
\begin{center}
\url{https://hkustconnect-my.sharepoint.com/:u:/g/personal/lchendh_connect_ust_hk/IQDhkh8dBHmDTI3p_VWvrq3tAdRPNZI18OJsaZ5jG-eUwOM}
\end{center}
The linked archive, \path{Supplementary_Material_S1.zip}, contains the three
exact source files above, the checker regression tests,
the reproduction script, compiler and operating-system information, the
complete optimized and ASan/UBSan outputs, and a manifest of SHA-256 digests.
\begin{samepage}
The archive and its internal manifest have the following SHA-256 digests:
\begin{center}
\begin{tabular}{ll}
\toprule
artifact & SHA-256\\
\midrule
archive
& \ttfamily c1779cb4c3edfc8ee38079b2a312ebdd\\[-1mm]
& \ttfamily 8ee99cafeee9644d6dc2c61bd4862d23\\
manifest
& \ttfamily 750390bc01622f782c7a52686bcd098a\\[-1mm]
& \ttfamily 37bc122621d040b22013f3cc8d55bb4b\\
\bottomrule
\end{tabular}
\end{center}
\end{samepage}
The URL is intended for editorial and referee access.  These hashes identify
the exact linked artifact, but the link does not provide permanent scholarly
access.  Before publication, the unchanged archive must be deposited as
journal supplementary material or in a DOI-bearing repository, and the
persistent identifier should be added to the published version.

\section*{Data and code availability}
No empirical data were generated or analyzed in this study.  The source code,
reproduction scripts, exact program outputs, and checksums supporting the
finite computer-assisted arguments are provided in Supplementary Material~S1.

\section*{Funding}
This research did not receive any specific grant from funding agencies in the
public, commercial, or not-for-profit sectors.

\section*{Declaration of competing interest}
The author declares no known competing financial interests or personal
relationships that could have appeared to influence the work reported in this
paper.

\section*{Declaration of generative AI and AI-assisted technologies in the
manuscript preparation process}
During the preparation of this work, the author used Kimi K3 solely to assist
with code writing, computational verification, manuscript drafting, and
language editing.  The author reviewed and edited the content as needed and
takes full responsibility for the content of the article.

\setlength{\bibsep}{2pt plus 1pt minus 1pt}
\bibliographystyle{plainnat}
\bibliography{references}

\end{document}